\documentclass[11pt,letterpaper,reqno]{amsart}
\usepackage[english]{babel}
\usepackage[T1]{fontenc}
\usepackage{verbatim}
\usepackage{palatino}
\usepackage{amsmath}
\usepackage{mathabx}
\usepackage{amssymb}
\usepackage{amsthm}
\usepackage{amsfonts}
\usepackage{overpic}
\usepackage{esint}
\usepackage{xcolor}
\usepackage{mathtools}
\usepackage{mathscinet}
\usepackage{enumerate}

\usepackage[colorlinks = true, linkcolor={red!80!black}, citecolor = {blue!60!black}]{hyperref}
\author{Pablo Shmerkin and Hong Wang}
\address{Department of Mathematics \\
The University of British Columbia \\
1984 Mathematics Road\\
Vancouver, BC, V6T 1Z2\\
Canada}
\email{pshmerkin@math.ubc.ca}

\address{Department of Mathematics, UCLA, Los Angeles, CA 90095, USA}
\email{hongwang@math.ucla.edu}

\title[Furstenberg sets and Bourgain's projection theorem]{Dimensions of Furstenberg sets and an extension of Bourgain's projection theorem}
\date{\today}

\thanks{P.S. is supported by an NSERC Discovery Grant}
\thanks{H.W. is supported by NSF Grant DMS-2055544}

\newcommand{\R}{\mathbb{R}}
\newcommand{\N}{\mathbb{N}}

\newcommand{\e}{\epsilon}

\DeclareMathOperator{\hdim}{dim_H}

\theoremstyle{plain}
\newtheorem{thm}{Theorem}
\newtheorem*{"thm"}{"Theorem"}

\newtheorem{lemma}[thm]{Lemma}

\newtheorem{cor}[thm]{Corollary}
\newtheorem{proposition}[thm]{Proposition}

\theoremstyle{definition}

\newtheorem{definition}[thm]{Definition}

\theoremstyle{remark}

\numberwithin{equation}{section}
\numberwithin{thm}{section}

\newcommand{\nref}[1]{(\hyperref[#1]{#1})}

\DeclareMathSymbol{\intop}  {\mathop}{mathx}{"B3}

\begin{document}

\begin{abstract}
We show that the Hausdorff dimension of $(s,t)$-Furstenberg sets is at least $s+t/2+\epsilon$, where $\epsilon>0$ depends only on $s$ and $t$. This improves the previously best known bound for $2s<t\le 1+\epsilon(s,t)$, in particular providing the first improvement since 1999 to the dimension of classical $s$-Furstenberg sets for $s<1/2$. We deduce this from a corresponding discretized incidence bound under minimal non-concentration assumptions, that simultaneously extends Bourgain's discretized projection and sum-product theorems. The proofs are based on a recent discretized incidence bound of T.~Orponen and the first author and a certain duality between $(s,t)$ and $(t/2,s+t/2)$-Furstenberg sets.
\end{abstract}

\maketitle


\section{Introduction and main results}

\subsection{Dimension of Furstenberg sets}

Let $s\in (0,1]$. We say that a set $K\subset\R^2$ is an $s$-Furstenberg set if for almost all directions $\theta\in S^1$ there is a line $\ell_{\theta}$ in direction $\theta$ such that $\hdim(K\cap \ell_{\theta})\ge s$. Motivated by work of Furstenberg and by the Szemer\'{e}di-Trotter Theorem in incidence geometry, T.~Wolff \cite{Wolff99} posed the problem of estimating the smallest possible dimension $\gamma(s)$ of an $s$-Furstenberg set. Using elementary geometric arguments, Wolff showed that
\[
\gamma(s) \ge \max(2s, s+1/2).
\]
Note that both bounds coincide for $s=1/2$. J.~Bourgain \cite{Bourgain03}, building up on work of N.~Katz and T.~Tao \cite{KatzTao01}, proved that $\gamma(1/2)>1+\e$ for some small universal constant $\e>0$; this is much deeper. Much more recently, T.~Orponen and the first author \cite{OrponenShmerkin21} established a similar improvement for $s\in (1/2,1)$:
\[
\gamma(s) \ge 2s+\e(s),
\]
where $\e(s)>0$ for $s\in (1/2,1)$. For the case $s<1/2$, the first author \cite{Shmerkin22} recently obtained a similar improvement, with the significant caveat that it involves only the \emph{packing} dimension of $s$-Furstenberg sets (which can be larger than Hausdorff dimension). In this paper, as a corollary of our main result we obtain the corresponding Hausdorff dimension improvement:

\begin{thm} \label{thm:classical-furst}
For every $s\in (0,1)$ there is $\e(s)>0$ such that every $s$-Furstenberg set $K$ satisfies
\[
\hdim(K) \ge s+\frac{1}{2}+\e(s).
\]
\end{thm}

Recently, there has been much interest in the following generalization of Furstenberg sets: we say that $K\subset\R^2$ is an \emph{$(s,t)$-Furstenberg set} if there is a family of lines $\mathcal{L}$ with $\hdim(\mathcal{L})\ge t$ such that
\[
\hdim(K\cap \ell) \ge s, \quad \ell\in\mathcal{L}.
\]
Since lines are a two-dimensional manifold, the Hausdorff dimension of $\mathcal{L}$ is well defined. Classical $s$-Furstenberg sets are, of course, $(s,1)$-Furstenberg sets. The central problem, initiated by U.~Molter and E.~Rela \cite{MolterRela12}, is to estimate $\gamma(s,t)$,  the smallest possible Hausdorff dimension of $(s,t)$-Furstenberg sets; this can be seen as a continuous analog of the Szemerédi-Trotter incidence bound. Recent works investigating this problem include \cite{LutzStull20, HSY22, DiBenedettoZahl21, DOV22, FuRen22}. The best currently known bounds are summarized as follows. Suppose first that $t\le 1+\e'(s,t)$ (where $\e'(s,t)$ is a small positive parameter). Then (see \cite{MolterRela12, LutzStull20, HSY22, OrponenShmerkin21})
\[
\gamma(s,t) \ge \left\{ \begin{array}{ll}
                  s+t & \text{if } t\le s \\
                  2s + \e(s,t) & \text{if }  s\le t\le 2s-\e'(s,t), \\
                  s+\frac{t}{2} & \text{if }  2s-\e'(s,t) \le t
                \end{array} \right..
\]
Suppose now that $t \ge 1+\e'(s,t)$. Then (see \cite{FuRen22})
\[
\gamma(s,t) \ge \left\{\begin{array}{ll}
                2s+t-1 & \text{if }  s+t\le 2\\
                  s + 1 & \text{if }  s+t\ge 2
                 \end{array}  \right..
\]
The bounds $s+t$, $s+1$ are sharp in the respective regimes, but the other bounds are not expected to be sharp. In this article we obtain a small improvement upon the $s+t/2$ bound:
\begin{thm} \label{thm:Furstenberg}
Given $s\in (0,1]$, $t\in (0,2]$, there is $\e(s,t)>0$ such that the following holds. Let $K$ be an $(s,t)$-Furstenberg set. Then
\[
\hdim(K) \ge s+\frac{t}{2}+\e(s,t).
\]
\end{thm}
Theorem \ref{thm:classical-furst} is an immediate corollary, taking $t=1$.

\subsection{Discretized incidence estimates and a strengthening of Bourgain's discretized projection and sum-product theorems}

Theorem \ref{thm:Furstenberg} is a consequence of the following discretized incidence estimate. See \S\ref{subsec:delta-s-sets} for the definition of $(\delta,s,C)$-sets of points and tubes.
\begin{proposition} \label{prop:main}
Given $s\in (0,1)$ and $t\in (s,2]$, there are $\e,\eta>0$ such that the following holds for small enough $\delta$.

Let $P\subset B^2(0,1)$ be a $(\delta,t,\delta^{-\e})$-set. For each $p\in P$, let $\mathcal{T}(p)$ be a $(\delta,s,\delta^{-\e})$-set of tubes through $p$ with $|\mathcal{T}(p)|\ge M$. Then the union $\mathcal{T}=\cup_{p\in P}\mathcal{T}(p)$ satisfies
\begin{equation} \label{eq:furst-gain}
|\mathcal{T}|\ge M \delta^{-(t/2+\eta)}.
\end{equation}
\end{proposition}

By duality between points and lines (see e.g. \cite[Theorem 3.2]{OrponenShmerkin21} and the discussion afterwards) we obtain the following corollary which more closely resembles the Furstenberg set problem. In the statement, $|P|_{\delta}$ stands for the $\delta$-covering number of $P$ (see \S\ref{subsec:notation}).
\begin{cor}
Given $0<s<1$, $s<t$, there are $\e,\eta>0$ such that the following holds for small enough dyadic $\delta$. Let $\mathcal{L}$ be a $(\delta,t,\delta^{-\e})$-set of lines intersecting $B^2(0,1)$. For each $\ell\in\mathcal{L}$, let $P(\ell)$ be a $(\delta,s,\delta^{-\e})$-set contained in $\ell^{(\delta)}$. Suppose $|P(\ell)|_{\delta}\ge M$ for all $\ell\in\mathcal{L}$. Then the union $P=\cup_{\ell\in\mathcal{L}} P(\ell)$ satisfies
\[
|P|_{\delta} \ge  M \delta^{-(t/2+\eta)}.
\]
\end{cor}
Note that $M\ge \delta^{-(t-\e)}$ and hence (up to changing the values of $\eta,\e$) we also have the conclusion $|\mathcal{T}|\ge \delta^{-(s+t/2+\eta)}$. In turn, by \cite[Lemma 3.3]{HSY22} this yields Theorem \ref{thm:Furstenberg}.

The non-concentration assumption on $\mathcal{T}(p)$ in Proposition \ref{prop:main} is quite mild, since $M$ can potentially be much larger than $\delta^{-s}$ (and a $\delta^{-\e}$ factor is also allowed). What about $P$? \emph{Some} non-concentration is needed, as the following standard example shows: if $P= B(x,r)$ and $\mathcal{T}(p)$ is the set of tubes through $p$ with slopes in a fixed $(\delta,s,C)$-set $\mathcal{S}$, then
\[
|\mathcal{T}| \sim |\mathcal{S}|_{\delta}\cdot \frac{r}{\delta} \sim |\mathcal{T}(p)|\cdot |P|_{\delta}^{1/2}.
\]
A similar estimate holds if $P$ is very dense in the union of a small number of $r$-balls. The next result asserts that under a minimal single-scale non-concentration assumption on $P$ that rules out this scenario, there is a gain over the ``trivial'' estimate of Corollary \ref{cor:nice-classical}:
\begin{thm} \label{thm:Furstenberg-Bourgain}
Given $s,u\in (0,1)$, there are $\e,\eta>0$ such that the following holds for small enough dyadic $\delta$.

Let $P\subset B^2(0,1)$ be set such that
\begin{equation} \label{eq:single-scale-non-conc}
\left|P\cap B\left(x,\delta |P|_{\delta}^{1/2}\right)\right|_{\delta} \le \delta^u |P|_{\delta}, \quad x\in B^2(0,1).
\end{equation}
For each $p\in P$, let $\mathcal{T}(p)$ be a $(\delta,s,\delta^{-\e})$-set of dyadic tubes through $p$ with $|\mathcal{T}(p)|\ge M$. Then the union $\mathcal{T}=\cup_{p\in P}\mathcal{T}(p)$ satisfies
\begin{equation} \label{eq:furst-gain}
|\mathcal{T}|\ge \delta^{-\eta} M |P|_{\delta}^{1/2}.
\end{equation}
\end{thm}

This result extends Proposition \ref{prop:main} (however, the proposition is used in the proof) and due to the minimal non-concentration assumption it provides new information even when $|P|_{\delta}\gg\delta^{-1}$. Perhaps more significantly, Theorem \ref{thm:Furstenberg-Bourgain} generalizes Bourgain's celebrated discretized projection theorem \cite[Theorem 3]{Bourgain10} (or even the refined version with single-scale non-concentration in \cite[Theorem 1.7]{Shmerkin22b}). Roughly speaking, Bourgain's Theorem corresponds to the special ``product'' case in which the slopes of the tubes in $\mathcal{T}(p)$ are (nearly) independent of $p$; see Section \ref{sec:Bourgain} for the details. This connection between Furstenberg sets and projections is well known, see e.g. \cite{Oberlin14} or \cite[\S 3.2]{OrponenShmerkin21}.  Bourgain's discretized projection theorem is used in the proof of \cite[Theorem 3.2]{OrponenShmerkin21}, which is in turn used to prove Theorem \ref{thm:Furstenberg-Bourgain}, so this does not provide a new proof of the projection theorem. Using a well-known argument of G.~Elekes, Theorem \ref{thm:Furstenberg-Bourgain} (or rather its dual formulation below) also easily recovers Bourgain's discretized sum-product theorem (\cite[Theorem 0.3]{Bourgain03}, \cite[Proposition 3.2]{BourgainGamburd08}). The details are sketched in Section \ref{sec:Bourgain}. It is worth noting that although Bourgain's discretized sum-product and projection theorems are closely connected to each other, deducing either from the other takes a substantial amount of work, while they are both rather direct corollaries of Theorem \ref{thm:Furstenberg-Bourgain}.

By duality between points and lines (again we refer to \cite[Theorem 3.2]{OrponenShmerkin21} for details), we have the following corollary of Theorem \ref{thm:Furstenberg-Bourgain}:
\begin{cor} \label{cor:Furstenberg-Bourgain}
Given $s,u\in (0,1)$, there are $\e,\eta>0$ such that the following holds for small enough dyadic $\delta$.

Let $\mathcal{T}\subset\mathcal{T}^{\delta}$ be a set of dyadic tubes such that
\[
\big| \{ T\in\mathcal{T}: T\subset \mathbf{T}\} \big| \le \delta^u |\mathcal{T}|
\]
for all $(\delta |\mathcal{T}|^{1/2})$-tubes $\mathbf{T}$.

For each $T\in \mathcal{T}$, let $P(T)\subset T$ be a $(\delta,s,\delta^{-\e})$-set with $|P(T)|_{\delta}\ge M$. Then $P=\cup_{T\in \mathcal{T}} P(T)$ satisfies
\[
|P|\ge \delta^{-\eta} M |T|^{1/2}.
\]
\end{cor}

\subsection{Sketch of proof}

Many ``$\e$-improvements'' in discretized geometry are obtained by showing that, in the absence of it, the relevant geometric object has a rigid structure that eventually is shown to contradict some previously known bounds (often involving some other ``$\e$-improvement''). This is also the approach we take here. By the simple elementary bounds in Lemmas \ref{lem:large-tube}--\ref{lem:small-tube}, one obtains the improved bound in Proposition \ref{prop:main} unless
\begin{align}
|P\cap T|_{\delta} &\approx |P|_\delta^{1/2},\quad T\in\mathcal{T},\nonumber \\
|P|_{\delta} &\approx \delta^{-t}. \label{eq:P-small}
\end{align}
(In this section the notation $\approx$ should be interpreted informally as ``up to small powers of $\delta$''). A first ingredient of the proof is showing that (after suitable refinements of $P$ and $\mathcal{T}$) one also gets the desired conclusion unless
\begin{equation} \label{eq:uniformity-cubes}
|P\cap T\cap Q|_{\delta} \approx |P\cap Q|_{\delta}^{1/2},
\end{equation}
for all $T\in\mathcal{T}$, all $\Delta$-squares $Q$ intersecting $P$ and a ``dense'' set of scales $\delta<\Delta<1$. To see this, we combine the elementary bounds applied to $P\cap Q$ (and a sub-system of tubes passing through $P\cap Q$) with an induction-on-scales mechanism from \cite{OrponenShmerkin21}, recalled as Proposition \ref{prop:induction-on-scales} below.

The relation \eqref{eq:uniformity-cubes} can be shown to imply that either one gets the improved bound we are seeking, or $P$ intersects each tube $T$ in a $(\delta,t/2)$-set. By (known) elementary arguments $\mathcal{T}$ is a $(\delta,s+t/2)$-set of tubes. Hence $P$ is a discretized $(t/2,s+t/2)$-Furstenberg set. But it was shown in \cite[Theorem 3.2]{OrponenShmerkin21} that this implies that $|P|_{\delta}\ge \delta^{-t-\eta}$ for some $\eta=\eta(t/2,s+t/2)>0$.  This contradicts \eqref{eq:P-small}, showing the impossibility of the rigid configuration described by \eqref{eq:uniformity-cubes} and hence establishing Proposition \ref{prop:main}. To our knowledge this dual relationship between $(s,t)$ and $(t/2,s+t/2)$-Furstenberg sets hadn't been noticed before.

We obtain Theorem \ref{thm:Furstenberg-Bourgain} by applying Proposition \ref{prop:main} to each scale in a multiscale decomposition of $P$ into ``non-trivial Frostman pieces'' that was established in \cite{Shmerkin22b}, and is recalled as Theorem \ref{thm:multiscale-decomposition} below. The scales are combined together by another application of Proposition \ref{prop:induction-on-scales}.

\section{Preliminaries}

\subsection{Notation}
\label{subsec:notation}

The notation $A\lesssim B$ or $A=O(B)$ stands for $A\le C\cdot B$ for some constant $C>0$; similarly for $A\gtrsim B$, $A\sim B$. The $\delta$-covering number of a set $X$ (in a metric space) is defined as the smallest number of $\delta$-balls needed to cover $X$, and is denoted $|X|_{\delta}$. The open $r$-neighbourhood of a set $X$ is denoted by $X^{(r)}$.

\subsection{$(\delta,s)$-sets of points and tubes}
\label{subsec:delta-s-sets}

Given $r\in 2^{-\N}$, we denote the family of (half-open) dyadic cubes of side-length in $\R^d$ by $\mathcal{D}_r$. The set of cubes in $\mathcal{D}_r$ intersecting a set $X$ is denoted by $\mathcal{D}_r(X)$.

In this article, we work with the following notion of discretization of sets of dimension $s$:
\begin{definition}[$(\delta,s,C)$-set]\label{def:delta-s-set}
Let $P \subset \R^{d}$ be a bounded nonempty set, $d \geq 1$. Let $\delta \in 2^{-\N}$, $0 \leq s \leq d$, and $C > 0$. We say that $P$ is a \emph{$(\delta,s,C)$-set} if
\begin{equation}\label{eq:delta-s-set}
|P \cap Q|_{\delta} \leq C \cdot |P|_{\delta} \cdot r^{s}, \qquad Q \in \mathcal{D}_{r}(\R^{d}), \, \delta \leq r \leq 1. \end{equation}
\end{definition}
If $P\subset\mathcal{D}_{\delta}$ (so $P$ is a family of dyadic cubes), we will abuse notation by identifying $P$ with $\cup P$, so it makes sense to speak of $(\delta,s,C)$-sets of dyadic cubes.

We also need to work with discretized families of tubes:
\begin{definition}[Dyadic $\delta$-tubes]\label{def:dyadicTubes} Let $\delta \in 2^{-\N}$. A \emph{dyadic $\delta$-tube} is a set of the form
\[
\mathbf{D}(p):= \{ (x,y): y=ax+b \text{ for some } (a,b)\in p\},
\]
where $p \in \mathcal{D}_{\delta}([-1,1)\times \R)$. The collection of all dyadic $\delta$-tubes is denoted
\[
\mathcal{T}^{\delta} := \{\mathbf{D}(p) : p \in \mathcal{D}_{\delta}([-1,1)\times \R)\}.
\]
A finite collection of dyadic $\delta$-tubes $\{\mathbf{D}(p)\}_{p \in \mathcal{P}}$ is called a $(\delta,s,C)$-set if $\mathcal{P}$ is a $(\delta,s,C)$-set.
\end{definition}

We remark that a dyadic $\delta$-tube is not exactly a $\delta$-neighbourhood of some line, but the intersection of
\[
T=\mathbf{D}([a,a+\delta]\times [b,b+\delta])
\]
with some fixed bounded set $B$ satisfies $\ell_{a,b}^{(c\delta)}\cap B\subset T\cap B\subset \ell_{a,b}^{(C\delta)}$, where $\ell_{a,b}=\{ y= (a+\delta/2)x+(b+\delta/2)\}$ and $c,C$ depend only
on $B$.

An elementary but important observation is that tubes in $\mathcal{T}^{\delta}$ that intersect a fixed  square $p\in\mathcal{D}_{\delta}$ are parametrized by their slope in a bilipschitz way. In particular, if $\mathcal{T}(p)\subset\mathcal{T}^{\delta}$ is a family of tubes intersecting $p\in\mathcal{D}_{\delta}$, then $\mathcal{T}(p)$ is a $(\delta,s,C)$-set if and only if the slopes of tubes in $\mathcal{T}(p)$ form a $(\delta,s,C'$) set for $C'\sim C$; see \cite[Corollary 2.12]{OrponenShmerkin21} for the precise statement.

\subsection{Elementary incidence bounds}

We now collect some elementary incidence bounds; they correspond in various ways to the lower bound $s+t/2$ for the dimension of $(s,t)$-Furstenberg sets. We state our bounds in terms of the following notion:
\begin{definition}
Fix $\delta\in 2^{-\mathbb{N}}$, $s\in [0,1]$, $C>0$,  $M\in\mathbb{N}$. We say that a pair $(P,\mathcal{T}) \subset \mathcal{D}_{\delta} \times \mathcal{T}^{\delta}$ s a \emph{$(\delta,s,C,M)$-nice configuration} if for every $p\in P$ there exists a $(\delta,s,C)$-set $\mathcal{T}(p) \subset\mathcal{T}$ with $ |\mathcal{T}(p)| = M$ and such that $T \cap p \neq\emptyset$ for all $T\in\mathcal{T}(p)$.
\end{definition}

\begin{lemma} \label{lem:large-tube}
Let $(P,\mathcal{T})$ be a $(\delta,s,C,M)$-nice configuration. Then for any $\delta$-tube $T$ (not necessarily in $\mathcal{T}$),
\[
|\mathcal{T}| \gtrsim C^{-1/s}\cdot  |T\cap P|_{\delta}\cdot M.
\]
\end{lemma}
\begin{proof}
We may assume $P$ is $\delta$-separated. Fix $p\in T\cap P$. Since $\mathcal{T}(p)$ is a $(\delta,s,C)$-set, there is a subset $\mathcal{T}'(p)\subset \mathcal{T}(p)$ such that $|\mathcal{T}'(p)|\ge |\mathcal{T}(p)|/2=M/2$ and each tube $T'\in\mathcal{T}'(p)$ makes an angle $\gtrsim C^{-1/s}$ with the direction of $T$. In turn, this implies that the sets $\mathcal{T}'(p)$, $p\in T\cap P$ have overlap $\lesssim C^{1/s}$. This gives the claim.
\end{proof}

\begin{lemma} \label{lem:small-tube}
Let $(P,\mathcal{T})$ be a $(\delta,s,C,M)$-nice configuration. Suppose $|T\cap P|\le K$ for all $T\in\mathcal{T}$. Then
\[
|\mathcal{T}| \ge K^{-1}\cdot |P|\cdot M.
\]
\end{lemma}
\begin{proof}
\[
|P|\cdot M = \sum_{p\in P} |\mathcal{T}_p| =  \sum_{T\in\mathcal{T}} | \{ p: T\in\mathcal{T}_p \} | \le \sum_{T\in\mathcal{T}} |T\cap P| \le |\mathcal{T}|\cdot K.
\]
\end{proof}

\begin{cor} \label{cor:nice-classical}
Let $(P,\mathcal{T})$ be a $(\delta,s,C,M)$-nice configuration. Then
\[
|\mathcal{T}|\gtrsim C^{-1/s} \cdot |P|^{1/2}\cdot M.
\]
Moreover, $\mathcal{T}$ contains a $(\delta,s+t/2, \log(1/\delta)^{O(1)} C^{1/s})$-set of $\delta$-tubes.
\end{cor}
\begin{proof}
If $|T\cap P|\ge |P|^{1/2}$ for some $T\in\mathcal{T}$, we apply Lemma \ref{lem:large-tube}, otherwise we apply Lemma \ref{lem:small-tube}. In any case we get the first claim.

Let $\mathcal{L}$ be the set of lines corresponding to tubes in $\mathcal{T}$ (or, equivalently, the $\delta$-neighbourhood of the central lines of tubes in $T$).  It follows from the first claim and two dyadic pigeonholings that the Hausdorff content of $\mathcal{L}$ satisfies
\[
\mathcal{H}_\infty^{s+t/2}(\mathcal{L})\gtrsim \log(1/\delta)^{-O(1)} C^{-1/s}.
\]
See e.g. the proof of \cite[Lemma 3.3]{HSY22} or \cite[Lemma 3.5]{OSW22} (in particular, \cite[Eq. (3.9)]{OSW22}). The conclusion then follows from the discrete version of Frostman's Lemma, \cite[Proposition A.1]{FasslerOrponen14} (which is stated in $\R^3$ but works just as well in the space of lines).
\end{proof}

\subsection{A multiscale incidence bound}

We next recall \cite[Proposition 5.2]{OrponenShmerkin21}. Fix two dyadic scales $0<\delta<\Delta\le 1$ and families $P_0\subset\mathcal{D}_{\delta}$ and $\mathcal{T}_0 \subset \mathcal{T}^{\delta}$. For $Q \in \mathcal{D}_{\Delta}$ and $\mathbf{T} \in \mathcal{T}^{\Delta}$, we denote
\[
P_0\cap Q = \{ p\in P_0: p\subset Q\} \quad \text{and} \quad \mathcal{T}_0 \cap \mathbf{T} := \{T \in \mathcal{T}_0 : T \subset \mathbf{T}\}.
\]
We also write
\begin{align*}
\mathcal{D}_{\Delta}(P_0) &= \{ Q\in\mathcal{D}_{\Delta}: P_0 \cap Q \neq\emptyset\},\\
\mathcal{T}^{\Delta}(\mathcal{T}_0) &= \{\mathbf{T} \in \mathcal{T}^{\Delta} : \mathcal{T}_0 \cap \mathbf{T} \neq \emptyset\}.
\end{align*}
In the next proposition, for $\Delta \in 2^{-\N}$ and $Q \in \mathcal{D}_{\Delta}$, the map $S_{Q} \colon \R^{2} \to \R^{2}$ is the homothety that maps $Q$ to the square $[0,1)^{2}$, and $S_Q(P_0)=\{ S_{Q}(p):p\in P_0\}$. Furthermore, the notation $A\lessapprox_{\delta} B$ stands for $ A\le \log(1/\delta)^C B$ for a constant $C>0$, and likewise for $A\approx_{\delta} B$.

\begin{proposition}[{\cite[Proposition 5.2]{OrponenShmerkin21}}] \label{prop:induction-on-scales}
Fix dyadic numbers $0<\delta<\Delta \leq 1$. Let $(P_0,\mathcal{T}_0)$ be a $(\delta,s,C,M)$-nice configuration. Then there exist sets $P\subset P_0$ and $\mathcal{T}(p)\subset\mathcal{T}_0(p)$, $p\in P$, such that denoting $\mathcal{T}=\bigcup_{p\in P} \mathcal{T}(p)$ the following hold:

\begin{enumerate}[(\rm i)]
\item \label{it:induction-i} $|\mathcal{D}_{\Delta}(P)| \approx_{\delta} |\mathcal{D}_{\Delta}(P_{0})|$ and $|P\cap Q| \approx_{\delta}|P_0\cap Q|$ for all $Q\in\mathcal{D}_{\Delta}(P)$.
\item \label{it:induction-ii} $|\mathcal{T}(p)|\gtrapprox_{\delta} |\mathcal{T}_0(p)|=M$ for $p\in P$.
\item \label{it:induction-iii} $(\mathcal{D}_{\Delta}(P),\mathcal{T}^{\Delta}(\mathcal{T}))$ is $(\Delta,s,C_\Delta,M_\Delta)$-nice for some $C_\Delta\approx_{\delta} C$ and $M_\Delta \geq 1$.
\item \label{it:induction-iv} For each $Q\in\mathcal{D}_{\Delta}(P)$ there exist $C_Q \approx_{\delta} C$, $M_Q\ge 1$,  and a family of tubes $\mathcal{T}_{Q}\subset\mathcal{T}^{\delta/\Delta}$ such that $(S_{Q}(P\cap Q),\mathcal{T}_Q)$ is $(\delta/\Delta,s,C_Q,M_Q)$-nice.
\end{enumerate}
Furthermore, the families  $\mathcal{T}_{Q}$ can be chosen so that
\begin{equation} \label{eq:lower-bound-T}
\frac{|\mathcal{T}_0|}{M} \gtrapprox_\delta \frac{|\mathcal{T}^{\Delta}(\mathcal{T})|}{M_\Delta}\cdot \left( \max_{Q \in \mathcal{D}_{\Delta}(P)}\frac{|\mathcal{T}_{Q}|}{M_Q} \right).
\end{equation}
\end{proposition}

\subsection{Uniformization}

Next, we recall a basic lemma asserting the existence of large uniform subsets.
See e.g. \cite[Lemma 7.3]{OrponenShmerkin21} for the proof.

\begin{definition}\label{def:uniformity}
Let $N \geq 1$, and let
\begin{displaymath} \delta = \Delta_{N} < \Delta_{N - 1} < \ldots < \Delta_{1} \leq \Delta_{0} = 1 \end{displaymath}
be a sequence of dyadic scales.  We say that a set $P\subset [0,1)^2$ is \emph{$(\Delta_j)_{j=1}^N$-uniform} if there is a sequence $(K_j)_{j=1}^N$ such that $|P\cap Q|_{\Delta_{j}} = K_j$ for all $j\in \{1,\ldots,N\}$ and all $Q\in\mathcal{D}_{\Delta_{j - 1}}(P)$.
\end{definition}

\begin{lemma} \label{lem:uniformization}
Given $P\subset [0,1)^{2}$ and a sequence $\delta = \Delta_{N} < \Delta_{N - 1} < \ldots < \Delta_{1} \leq \Delta_{0} = 1$ of dyadic numbers, $N \geq 1$, there is a $(\Delta_j)_{j=1}^N$-uniform set $P'\subset P$ such that
\begin{equation}\label{form12}
|P'|_\delta \ge  \left(4 N^{-1}\log(1/\delta)\right)^{-N} |P|_\delta. \end{equation}
\end{lemma}

Note that if the number $N$ of scales is independent of $\delta$, then lower bound on $|P'|_{\delta}$ can be simplified as
\[
|P'|_{\delta} \ge C_N^{-1} \cdot \log(1/\delta)^{-C_N} \cdot |P|_{\delta},
\]
with $C_N$ independent of $\delta$.

\subsection{A multiscale decomposition}

To conclude this section, we recall the multiscale decomposition into ``Frostman pieces'' of uniform sets that satisfy a single-scale non-concentration assumption provided by \cite[Theorem 4.1]{Shmerkin22b}.
\begin{thm} \label{thm:multiscale-decomposition}
For every $u>0$ and $\e>0$ there are $\xi=\xi(u)>0$ and $\tau=\tau(\e)>0$ such that the following holds for all sufficiently small $\rho\le \rho_0(\e)$ and  $n \ge n_0(\rho,\e)$:

Let $P$ be a $(\rho^j)_{j=1}^n$-uniform set and write $\delta=\rho^n$. Suppose
\begin{equation} \label{eq:single-scale-Frostman}
\left|P \cap B(x,\delta |P|_{\delta}^{1/2})\right|_{\delta} \le \delta^u |P|_{\delta} \quad \text{for all }x.
\end{equation}
Then there exists a collection of dyadic scales
\[
\delta= \Delta_N < \Delta_{N-1}<\ldots<\Delta_1 <\Delta_0=1, \quad N\le N_0(\e).
\]
each of which is a power of $\rho$, and numbers $\alpha_0,\ldots,\alpha_{N-1}\in [0,2]$ such that, denoting $\lambda_j = \Delta_{j+1}/\Delta_j$, the following hold:

\begin{enumerate}[(\rm i)]
  \item \label{it:multiscale-dec-i} For each $j$ and each $Q\in \mathcal{D}_{\Delta_{j}}(P)$,
\[
  |P\cap Q\cap B(x,r\Delta_{j})|_{\Delta_{j+1}} \le \lambda_j^{-\e} \cdot r^{\alpha_j} \cdot |P\cap Q|_{\Delta_{j+1}}
\]
for all $x\in B^2(0,1)$ and all $r\in [\lambda_j,1]$.
  \item \label{it:multiscale-dec-ii}
  \[
  \sum \big\{ \alpha_j \log(1/\lambda_j): \lambda_j\le\delta^{\tau}\big\} \ge \log |P|_{\delta} - 2\e \log(1/\delta). 
  \]
  \item \label{it:multiscale-dec-iii}
  \[
  \sum\big\{ \log(1/\lambda_j) : \alpha_j \in [\xi,d-\xi] \text{ and } \lambda_j \le \delta^{\tau}  \big\} \ge \xi \log(1/\delta).
  \]
\end{enumerate}
\end{thm}

This is just (a slightly weaker version of) \cite[Theorem 4.1]{Shmerkin22b}, although stated using different notation: the measure $\mu$ there corresponds to normalized Lebesgue measure on $P(\delta)$ (or the union of $\delta$-squares intersecting $\delta$); $\rho$ corresponds to $2^{-T}$ in \cite{Shmerkin22b}; the scales $\Delta_j$ correspond to both $2^{-TA_i}$ and $2^{-TB_i}$. The last two claims only concern scale intervals $[\Delta_{j+1},\Delta_j]$ corresponding to $[2^{-T B_i}, 2^{-TA_i}]$, while for $[\Delta_{j+1},\Delta_j] = [2^{-T B_{i+1}}, 2^{-T A_i}]$ we simply take $\alpha_j=0$.


\section{Proof of Proposition \ref{prop:main}}

In this section we prove Proposition \ref{prop:main}. Note that there is no loss of generality in assuming that $P$ is a union of dyadic $\delta$-squares and the tubes in $\mathcal{T}(p)$ are dyadic $\delta$-tubes intersecting $p$.

The parameter $\e$ should be thought of being much smaller than $\eta$ (and will be chosen after $\eta$). Both $\e,\eta$ will ultimately be chosen in terms of $s,t$ only. We let $N=\lceil \eta^{-1}\rceil$, so that $N=N(s,t)$. Let $\rho=\delta^{1/N}$; without loss of generality, $\rho$ is dyadic.

In the course of this proof, $A\lessapprox B$ stands for $A\le C(s,t) \delta^{-C(s,t)\e} B$. Likewise, a $(\delta,u)$-set is a $(\delta,u,C)$-set for $C\lessapprox 1$.

Replacing $P$ by its $(\rho^j)_{j=1}^N$-uniformization (given by Lemma \ref{lem:uniformization}), we may assume that $P$ is $(\rho^j)_{j=1}^N$-uniform. Note that the uniformization is still a $(\delta,t)$-set.

We will construct sequences $P_j\subset\mathcal{D}_\delta$, $\mathcal{T}_j\subset\mathcal{T}^{\delta}$, $j=0,1,\ldots,N$ with the following properties:
\begin{enumerate}[(a)]
  \item \label{it:a} $P_j$ is  $(\rho^j)_{j=1}^N$-uniform.
  \item \label{it:b} $P_{j+1}\subset P_j$, $P_0\subset P$ and $|P_{j+1}|\gtrapprox |P_j|$, $|P_0|\gtrapprox|P|$. In particular, $P_N$ is a $(\delta,t)$-set.
  \item \label{it:c} $\mathcal{T}_0(p)=\mathcal{T}(p)$, $\mathcal{T}_{j+1}(p)\subset\mathcal{T}_{j}(p)$ and $M_{j+1}\sim|\mathcal{T}_{j+1}(p)|\gtrapprox |\mathcal{T}_{j}(p)|$ for $p\in P_{j+1}$ and some $M_{j+1}$. In particular, each $\mathcal{T}(p)$, $p\in P_N$ is a $(\delta,t)$-set.
  \item \label{it:d} For each $j$, any $Q\in\mathcal{D}_{\rho^j}(P_j)$ and any $\delta$-tube $T$,
  \[
  |\mathcal{T}| \gtrapprox M \cdot |P_j|_{\rho^j}^{1/2} \cdot |T\cap P_j\cap  Q|_\delta.
  \]
\end{enumerate}

Recall that $|\mathcal{T}(p)|\geq M$ for each $p\in P$. Pigeonhole a dyadic number $M_0$ such that $|\mathcal{T}(p)|\sim M_0$ for all $p\in P'_0\subset P$ with $|P'_0|\gtrapprox |P|$. We let $P_0$ be the $(\rho^j)_{j=1}^N$-uniformization of $P'_0$ given by Lemma \ref{lem:uniformization}. Then $P_0$ is a $(\delta,t)$-set, and we take $\mathcal{T}_0(p)=\mathcal{T}(p)$ for $p\in P_0$.

Once $P_j,\mathcal{T}_j$ are defined, we let $P'_{j+1}$, $\mathcal{T}'_{j+1}$ be the objects provided by Proposition \ref{prop:induction-on-scales} applied to $(P_{j},\mathcal{T}_j)$ at scale $\Delta=\rho^{j+1}$. It follows from Proposition \ref{prop:induction-on-scales}\eqref{it:induction-i} and the regularity of $P_j$ that $P'_{j+1}\subset P_j$ and  $|P'_{j+1}|\gtrapprox P_j$. Pigeonhole a number $M_{j+1}$ such that $|\mathcal{T}'_{j+1}(p)|\sim M_{j+1}$ for all $p\in P''_{j+1}\subset P'_{j+1}$, where $|P''_{j+1}|\gtrapprox |P'_{j+1}|$. Finally, let $P_{j+1}$ be the $(\rho^j)_{j=1}^N$-uniformization of $P''_{j+1}$ and $\mathcal{T}_{j+1}(p)=\mathcal{T}'_{j+1}(p)$ for $p\in P_{j+1}$.

Properties \eqref{it:a}--\eqref{it:b} hold by construction. Property \eqref{it:c} follows from Proposition \ref{prop:induction-on-scales}\eqref{it:induction-ii}. To see Property
\eqref{it:d}, let $C_{\Delta}$, $M_{\Delta}$, $\mathcal{T}_Q$, $C_Q$, $M_Q$ be the objects provided by Proposition \ref{prop:induction-on-scales}\eqref{it:induction-iii}--\eqref{it:induction-iv}. By Corollary \ref{cor:nice-classical},
\[
|\mathcal{T}^{\Delta}(\mathcal{T}'_{j+1})| \gtrapprox M_\Delta \cdot |P_{j+1}|^{1/2},
\]
and by Lemma \ref{lem:large-tube} and rescaling,
\[
|\mathcal{T}_Q| \gtrapprox M_Q \cdot |T\cap P_j \cap Q|_{\delta} .
\]
Putting these facts together with  \eqref{eq:lower-bound-T}, we see that Property \eqref{it:d} holds.

We pause to observe that $(\mathcal{P}_N,\mathcal{T}_N)$ is a $(\delta,s,C_N,M_N)$-nice configuration where $C_N\lessapprox 1$ and $M_N\gtrapprox M$.

We now consider several cases. Suppose first that there are $T\in\mathcal{T}_N$, $j\in\{1,\ldots,N\}$ and $Q\in\mathcal{D}_{\rho^j}(P_N)$  such that
\[
|T\cap P_N\cap Q|_{\delta} \ge \delta^{-2\eta} \cdot |P_N\cap Q|_{\delta}^{1/2}.
\]
Then, by \eqref{it:d}, the uniformity of $P_N$, and the fact that $P_j\supset P_N$ is a $(\delta,t)$-set, we see that \eqref{eq:furst-gain} holds if $\e$ is small enough in terms of $s,t,\eta$.

Hence, we assume from now on that
\begin{equation} \label{eq:tube-upper}
|T\cap P_N\cap Q|_{\delta} \le \delta^{-2\eta} \cdot |P_N\cap Q|_{\delta}^{1/2},
\end{equation}
for  $j\in\{1,\ldots,N\}$, $Q\in\mathcal{D}_{\rho^j}(P_N)$, and $T\in\mathcal{T}_N$.

We consider two further sub-cases. Suppose first that for at least half of the squares $p$ in $P_N$, at least half of the tubes $T\in\mathcal{T}_N(p)$ satisfy
\[
|T\cap P_N|_{\delta} \le \delta^{2\eta}\cdot |P_N|_{\delta}^{1/2}.
\]
Then Lemma \ref{lem:small-tube} (applied to suitable restrictions of $P_N$ and $\mathcal{T}_N$) yields \eqref{eq:furst-gain}.

We can then assume that
\begin{equation} \label{eq:tube-lower}
|T\cap P_N|_{\delta} \ge \delta^{2\eta}\cdot |P_N|_{\delta}^{1/2}\quad\text{for all }T\in\mathcal{T}'_N,
\end{equation}
where $\mathcal{T}'_N=\cup_{p\in P'_N} \mathcal{T}'(p)$, with $|P'_N|\ge |P_N|/2$ and $|\mathcal{T}'_N(p)|\ge |\mathcal{T}_N(p)|/2$. It then follows from \eqref{eq:tube-upper} that, for any $Q\in\mathcal{D}_{\rho^n}(P_N)$ and $T\in\mathcal{T}'_N$,
\begin{align*}
|T\cap P_N\cap Q|_\delta &\le \delta^{-4\eta} \left(|P_N|_{\delta}^{-1/2}|P_N\cap Q|_{\delta}^{1/2}\right) |T\cap P_N|_{\delta}\\
&\le \delta^{-5\eta} (\rho^{j})^{t/2} |T\cap P_N|_{\delta},
\end{align*}
using that $P_N$ is a $(\delta,t)$-set and taking $\e$ sufficiently small in terms of $\eta,s,t$.

Recalling that $N=\lceil \eta^{-1}\rceil$, we deduce that for any $r$-ball $B_r$ with $\rho^{j+1}\le r<\rho^j$,
\[
|T\cap P_N\cap B_r|_\delta \lesssim \delta^{-5\eta} (\rho^j)^{t/2}|T\cap P_N|_\delta \le \delta^{-6\eta} r^{t/2}|T\cap P_N|_\delta,
\]
so that $T\cap P_N$ is a $(\delta,t/2,\delta^{-7\eta})$-set for each $T\in\mathcal{T}'$.

Now, taking $\e$ small enough in terms of $s,t,\eta$, we deduce from Corollary \ref{cor:nice-classical} that $\mathcal{T}'$ contains a $(\delta,s+t/2,\delta^{-\eta})$-set $\mathcal{T}''$. Then $\{ T\cap P_N:T\in \mathcal{T}''\}$ satisfies the assumptions of \cite[Theorem 3.2]{OrponenShmerkin21} (with $t/2$ in place of $s$ and $s+t/2$ in place of $t$), provided that $\eta$ and $\delta$ are taken small enough in terms of $s,t$ only. Applying \cite[Theorem 3.2]{OrponenShmerkin21}, we conclude that  $|P_N|_{\delta}> \delta^{-t-3\eta}$ (again assuming $\eta=\eta(s,t)$ is small enough). The first claim of Corollary \ref{cor:nice-classical} then yields \eqref{eq:furst-gain}.

\section{Proof of Theorem  \ref{thm:Furstenberg-Bourgain}}

By iterating Proposition \ref{prop:induction-on-scales}, we obtain the follow multiscale version.
\begin{cor} \label{cor:induction-on-multiscales}
Fix $N\ge 2$ and dyadic numbers
\[
0<\delta=\Delta_N<\Delta_{N-1}<\ldots<\Delta_1<\Delta_0 = 1.
\]
Let $(P_0,\mathcal{T}_0)$ be a $(\delta,s,C,M)$-nice configuration. Then there exists a set $P\subset P_0$  such that the following hold:
\begin{enumerate}[(\rm i)]
\item \label{it:induction-multi-i} $|\mathcal{D}_{\Delta_j}(P)| \approx_{\delta} |\mathcal{D}_{\Delta_j}(P_{0})|$ and $|P\cap Q| \approx_{\delta}|P_0\cap Q|$ for all $j\in [1,N]$ and all $Q\in\mathcal{D}_{\Delta_j}(P)$.
\item \label{it:induction-multi-ii} For each $j\in [0,N-1]$ and each $Q\in\mathcal{D}_{\Delta_j}(P)$ there exist $C_Q \approx_{\delta} C$, $M_Q\ge 1$,  and a family of tubes $\mathcal{T}_{Q}\subset\mathcal{T}^{\Delta_{j+1}/\Delta_j}$ such that $(S_{Q}(P\cap Q),\mathcal{T}_Q)$ is $(\Delta_{j+1}/\Delta_j,s,C_Q,M_Q)$-nice.
\end{enumerate}
Furthermore, the families  $\mathcal{T}_{Q}$ can be chosen so that if $Q_j\in\mathcal{D}_{\Delta_j}(P)$, then
\begin{equation} \label{eq:lower-bound-T-multi}
\frac{|\mathcal{T}_0|}{M} \gtrapprox_\delta \prod_{j=0}^{N-1}  \frac{|\mathcal{T}_{Q_j}|}{M_{Q_j}}.
\end{equation}
All the constants implicit in the $\approx_{\delta}$ notation are allowed to depend on $N$.
\end{cor}
\begin{proof}
We proceed by induction in $N$. The case $N=2$ follows from Proposition \ref{prop:induction-on-scales}. Suppose the claim holds for $N$ and let us verify it for $N+1$. First apply Proposition \ref{prop:induction-on-scales} with $\delta=\Delta_{N+1}$ and $\Delta=\Delta_{N}$.  Let $P'$, $\mathcal{T}$ be the resulting objects. Property \eqref{it:induction-multi-ii} holds (at the moment for $P'$) for $j=N$ thanks to Proposition \ref{prop:induction-on-scales}\eqref{it:induction-iv}. We then apply the inductive assumption to $(\mathcal{D}_{\Delta_{N}}(P'),\mathcal{T}^{\Delta_N}(\mathcal{T}))$, which is legitimate by Proposition \ref{prop:induction-on-scales}\eqref{it:induction-iii}. This yields a set $P''\subset \mathcal{D}_{\Delta_{N}}$; we define
\[
P = \bigcup_{Q\in \mathcal{D}_{\Delta_{N}}(P'')} P'\cap Q.
\]
This ensures that $S_{Q}(P\cap Q)$, when viewed at scale $\Delta_{j+1}/\Delta_{j}$, equals $P''$ for $j<N$ and $P'$ for $j=N$, and so Property \eqref{it:induction-multi-ii} holds for all $j\in [0,N]$.  Finally, \eqref{eq:lower-bound-T-multi} holds thanks to \eqref{eq:lower-bound-T} and the inductive assumption.
\end{proof}

\begin{proof}[Proof of Theorem \ref{thm:Furstenberg-Bourgain}]
We will eventually take $\eta=\eta(s,u)$ and $\e=\e(\eta,s,u)\ll \eta$. Fix $\rho=\rho(\e)\in (0,1)$ small enough that the conclusion of Theorem \ref{thm:multiscale-decomposition} holds, and
\[
\frac{\log(4\log(1/\rho))}{\log(1/\rho)} < \e.
\]
Then, for $\delta=\rho^n$, the $(\rho^j)_{j=1}^n$-uniformization $P'$ of $P$ given by Lemma \ref{lem:uniformization} satisfies $|P'|\ge \delta^{\e} |P|$. We assume from now on that $\delta=\rho^n$, where $n$ is taken sufficiently large in terms of all other parameters. Take $\e<\min(u/2,\eta/2)$. Then $P'$ satisfies \eqref{eq:single-scale-non-conc} with $u/2$ in place of $u$, and the conclusion \eqref{eq:furst-gain} for $P'$ implies it for $P$ with $\eta/2$ in place of $\eta$. Hence we assume from now on that $P$ is $(\rho^j)_{j=1}^n$-uniform.

We apply Theorem \ref{thm:multiscale-decomposition} to $P$. Let $\xi=\xi(u)$, $\tau=\tau(\e)$ be the numbers provided by the theorem, and let $(\Delta_j)_{j=0}^N$ and $(\alpha_j)_{j=1}^N$ be the scales and exponents corresponding to $P$.

Let $\lambda_j = \Delta_{j+1}/\Delta_j$. We apply Corollary \ref{cor:induction-on-multiscales} to $(P,\mathcal{T})$ and the scales $(\Delta_j)$. Let $P'\subset P$ be the resulting set. Since $N\le N_0(\e)$, the notation $A\lessapprox_{\delta} B$ in the statement of Corollary \ref{cor:induction-on-multiscales} translates to $A\le \log(1/\delta)^{C(\e)} B$; in particular, $A \le \delta^{\e\tau/2} B$  if $\delta$ is small enough in terms of $\e$. With these remarks, the $(\Delta_j)_{j=1}^N$-uniformity of $P$, Corollary \ref{cor:induction-on-multiscales}\eqref{it:induction-multi-i} and Theorem \ref{thm:multiscale-decomposition}\eqref{it:multiscale-dec-i} show that if $Q\in\mathcal{D}_{\Delta_j}(P')$, then
$S_Q(P'\cap Q)$ is a $(\lambda_j, \alpha_j, \lambda_j^{-\e})$-set whenever $\lambda_j \le \delta^\tau$.

Let
\begin{align*}
\mathcal{N} &= \big\{ j:  \lambda_j \le \delta^{\tau},\alpha_j\notin [\xi,2-\xi] \big\},\\
\mathcal{G} &= \big\{ j:  \lambda_j \le \delta^{\tau},\alpha_j\in [\xi,2-\xi] \big\}.
\end{align*}
(Here $\mathcal{N}$ stands for ``normal'' and $\mathcal{G}$ for ``good'' scales.)  It follows from Corollary \ref{cor:induction-on-multiscales}\eqref{it:induction-multi-ii} combined with Corollary \ref{cor:nice-classical} and Proposition \ref{prop:main} that, for $Q\in\mathcal{D}_{\Delta_j}(P')$,
\begin{align*}
j \in \mathcal{N} &\Longrightarrow |\mathcal{T}_Q|\ge (1/\log\delta)^{C(\e,s)}\cdot  M_Q\cdot |S_Q(P\cap Q)|_{\lambda_j}^{1/2}, \\
j \in \mathcal{G} &\Longrightarrow |\mathcal{T}_Q|\ge (1/\log\delta)^{C(\e,s)}\cdot  M_Q\cdot \lambda_j^{-(\alpha_j/2-\eta)},
\end{align*}
where $\eta=\eta(\xi,s)=\eta(u,s)>0$. It is indeed possible to take a value of $\eta$ uniformly over $t\in [\xi,2-\xi]$ because the value of $\eta$ in Proposition \ref{prop:main} is robust under perturbations of $t$. Note that since $\tau=\tau(\e)$, if $\delta$ is small enough in terms of $\e$, and $j\in\mathcal{G}$, then $\lambda_j \le \delta^{\tau}$ is small enough that Proposition \ref{prop:main} is indeed applicable. It follows from Theorem \ref{thm:multiscale-decomposition}\eqref{it:multiscale-dec-i} that
\[
|S_Q(P\cap Q)|_{\lambda_j} \ge \lambda_j^{\e-\alpha_j}.
\]
Combining these facts with the conclusion \eqref{eq:lower-bound-T-multi} and the trivial bound $|\mathcal{T}_Q|\ge M_Q$ (applied at scales outside $\mathcal{N}\cup \mathcal{G}$), we obtain
\begin{align*}
\frac{|\mathcal{T}|}{M} &\ge \left(\prod_{j\in\mathcal{N}} \lambda_j^{\e/2}\lambda_j^{-\alpha_j/2} \right)\left( \prod_{j\in\mathcal{G}} \lambda_j^{-\eta}\lambda_j^{-\alpha_j/2} \right) \\
&\ge \delta^{\e} \left(\prod_{j\in\mathcal{N}\cup\mathcal{G}} \lambda_j^{-\alpha_j/2}\right) \left( \prod_{j\in\mathcal{G}} \lambda_j^{-\eta}\right) \\
&\ge \delta^{3\e} \cdot |P|_{\delta}^{1/2}\cdot \delta^{-\xi\eta},
\end{align*}
using Theorem \ref{thm:multiscale-decomposition}\eqref{it:multiscale-dec-ii}--\eqref{it:multiscale-dec-iii} for the last inequality. Taking $\e<\xi\eta/6$, this gives the claim with $\xi\eta/2$ in place of $\eta$.

\end{proof}

\section{Connection with Bourgain's Projection Theorem}
\label{sec:Bourgain}

We conclude by showing how Theorem \ref{thm:Furstenberg-Bourgain} has as corollaries both Bourgain's discretized projection and sum-product theorems. We start with the former. Let $\Pi_s(x,y)=x-sy$.
\begin{thm}[Bourgain's discretized projection theorem (\cite{Bourgain10, Shmerkin22b})]
Given $s,u\in (0,1)$ there are $\e,\eta>0$ such that the following hold.

Let $P\subset B^2(0,1)$ satisfy
\[
\left|P\cap B\left(x,\delta |P|_{\delta}^{1/2}\right)\right|_{\delta} \le \delta^u |P|_{\delta}, \quad x\in B^2(0,1).
\]
Let $S\subset [1,2]$ be a $(\delta,s,\delta^{-\e})$-set. Then there is $s\in S$ such that
\[
|\Pi_s(P')|_{\delta} \ge \delta^{-\eta} \cdot |P|_{\delta}^{1/2}\quad\text{for all }P'\subset P, |P'|_{\delta}\ge \delta^{\e}|P|_{\delta}.
\]
\end{thm}
\begin{proof}
The argument is standard. Suppose the claim does not hold. Hence, for each $s\in S$ there is a set $P_s\subset P$ with $|P_s|_{\delta}\ge \delta^{\e}|P|_{\delta}$ such that
\[
|\Pi_s(P_s)|_{\delta} \le \delta^{-\eta} \cdot |P|^{1/2}.
\]
The set $X=\{ (p,s): p\in P_s\}$ has size $\ge \delta^{\e}|P||S|$, hence there is a set $P'\subset P$ with $|P'|_{\delta}\gtrsim \delta^{\e}|P|_{\delta}$ such that $|S_p|\gtrsim \delta^{\e}|S|$, where $S_p=\{ s: p\in P_s\}$.

Given a tube $T=\mathbf{D}([a,a+\delta]\times [b,b+\delta])$ (recall Definition \ref{def:dyadicTubes}), we let $\sigma(T)=[a,a+\delta]$ be the corresponding slope interval.
For each $p\in P'$ let $\mathcal{T}_p$ be the set of dyadic tubes through $p$ such that $\sigma(T)\cap S_p\neq \emptyset$.

If $\e<u/2$, then $P'$ still satisfies the single-scale non-concentration assumption (with $u/2$ in place of $u$). Since $S_p$ is a $(\delta,s,O(\delta^{-2\e}))$-set, so is $\mathcal{T}_p$. Also, $|\mathcal{T}_p|\gtrsim \delta^{\e}|S|$. Hence if $\e>0$ is small enough in terms of $s,u$, we can apply Theorem \ref{thm:Furstenberg-Bourgain} to obtain (for $\mathcal{T}=\cup_{p\in P'}\mathcal{T}_p$)
\[
|\mathcal{T}| \gtrsim \delta^{-\eta'} \cdot \delta^{\e}|S|\cdot |P'|_{\delta}^{1/2} \ge \delta^{2\e-\eta'} \cdot |S|\cdot |P|_{\delta}^{1/2},
\]
where $\eta'>0$ depends on $s,u$ only. Taking $\eta'>4\e$, this implies that there is $s\in S$ such that there are $\gtrsim \delta^{-\eta'/2}\cdot |P|_{\delta}^{1/2}$ tubes $T\in\mathcal{T}$ with $s\in \sigma(T)$. All of these tubes intersect $P_s$ by construction. We conclude that
\[
|\Pi_s P_s|_{\delta} \gtrsim \delta^{-\eta'/2} |P|_{\delta}^{1/2}
\]
which is a contradiction if we take $\eta=\eta'/3$.
\end{proof}

We turn to the discretized sum-product problem. We have the following corollary of Theorem \ref{thm:Furstenberg-Bourgain}:
\begin{cor}
Given $s,u\in (0,1)$, there are $\e,\eta>0$ such that the following holds for all small enough $\delta$.

Let $A, B_1, B_2\subset [1,2]$ satisfy the following: $A$ is a $(\delta,s,\delta^{-\e})$-set, and $B_1, B_2$ satisfy the single-scale non-concentration bound
\[
|B_i\cap [a,a+\delta|B_i|_{\delta}]|_{\delta} \le \delta^{u}|B_i|, \quad a\in [1,2].
\]
Then
\begin{equation} \label{eq:sum-product}
|A+B_1|_{\delta}|A\cdot B_2|_{\delta} \ge \delta^{-\eta} |A|_{\delta} |B_1|_{\delta}^{1/2} |B_2|_{\delta}^{1/2}.
\end{equation}
\end{cor}
Taking $A=B_1=B_2$, one immediately recovers Bourgain's discretized sum-product theorem, even under the weak non-concentration assumption of Bourgain-Gamburd \cite[Proposition 3.2]{BourgainGamburd08}.

To prove the corollary, consider (as in \cite{Elekes97}) the set $P= (A+B_1)\times (A\cdot B_2)$. For each $(b_1,b_2)\in B_1\times B_2$, the set $P$ intersects the line $\ell_{b_1,b_2}=\{ x+b_1, b_2 x: x\in\R\}$ in an affine copy of $A$. Using that $(b_1,b_2)\to \ell_{b_1,b_2}$ is bilipschitz, it is routine to verify that this configuration satisfies the assumptions of Corollary \ref{cor:Furstenberg-Bourgain}, with $2u$ in place of $u$ (considering for each $(b_1,b_2)$ the $\delta$-dyadic tube that contains $\ell_{b_1,b_2}\cap [0,2]^2$). See e.g. \cite[\S 6.3]{DOV22} for details of adapting Elekes' argument to the discretized setting. The conclusion of Corollary \ref{cor:Furstenberg-Bourgain} is then precisely \eqref{eq:sum-product}.


\begin{thebibliography}{10}

\bibitem{Bourgain03}
Jean Bourgain.
\newblock On the {E}rd{\H{o}}s-{V}olkmann and {K}atz-{T}ao ring conjectures.
\newblock {\em Geom. Funct. Anal.}, 13(2):334--365, 2003.

\bibitem{Bourgain10}
Jean Bourgain.
\newblock The discretized sum-product and projection theorems.
\newblock {\em J. Anal. Math.}, 112:193--236, 2010.

\bibitem{BourgainGamburd08}
Jean Bourgain and Alex Gamburd.
\newblock On the spectral gap for finitely-generated subgroups of {$\rm
  SU(2)$}.
\newblock {\em Invent. Math.}, 171(1):83--121, 2008.

\bibitem{DOV22}
Damian D{\polhk a}browski, Tuomas Orponen, and Michele Villa.
\newblock Integrability of orthogonal projections, and applications to
  {F}urstenberg sets.
\newblock {\em Adv. Math.}, 407:Paper No. 108567, 34, 2022.

\bibitem{DiBenedettoZahl21}
Daniel Di~Benedetto and Joshua Zahl.
\newblock New estimates on the size of $(\alpha,2\alpha)$-{F}urstenberg sets.
\newblock Preprint, arXiv:2112.08249, 2021.

\bibitem{Elekes97}
Gy\"{o}rgy Elekes.
\newblock On the number of sums and products.
\newblock {\em Acta Arith.}, 81(4):365--367, 1997.

\bibitem{FasslerOrponen14}
Katrin F\"{a}ssler and Tuomas Orponen.
\newblock On restricted families of projections in {$\Bbb R^3$}.
\newblock {\em Proc. Lond. Math. Soc. (3)}, 109(2):353--381, 2014.

\bibitem{FuRen22}
Yuqiu Fu and Kevin Ren.
\newblock Incidence estimates for $\alpha$-dimensional tubes and
  $\beta$-dimensional balls in $\mathbb{R}^2$.
\newblock Preprint, arXiv:2111.05093, 2022.

\bibitem{HSY22}
Korn\'{e}lia H\'{e}ra, Pablo Shmerkin, and Alexia Yavicoli.
\newblock An improved bound for the dimension of
  {$(\alpha,2\alpha)$}-{F}urstenberg sets.
\newblock {\em Rev. Mat. Iberoam.}, 38(1):295--322, 2022.

\bibitem{KatzTao01}
Nets~Hawk Katz and Terence Tao.
\newblock Some connections between {F}alconer's distance set conjecture and
  sets of {F}urstenburg type.
\newblock {\em New York J. Math.}, 7:149--187, 2001.

\bibitem{LutzStull20}
Neil Lutz and Donald~M. Stull.
\newblock Bounding the dimension of points on a line.
\newblock {\em Inform. and Comput.}, 275:104601, 15, 2020.

\bibitem{MolterRela12}
Ursula Molter and Ezequiel Rela.
\newblock Furstenberg sets for a fractal set of directions.
\newblock {\em Proc. Amer. Math. Soc.}, 140(8):2753--2765, 2012.

\bibitem{Oberlin14}
Daniel~M. Oberlin.
\newblock Some toy {F}urstenberg sets and projections of the four-corner
  {C}antor set.
\newblock {\em Proc. Amer. Math. Soc.}, 142(4):1209--1215, 2014.

\bibitem{OrponenShmerkin21}
Tuomas Orponen and Pablo Shmerkin.
\newblock On the {H}ausdorff dimension of {F}urstenberg sets and orthogonal
  projections in the plane.
\newblock Preprint, arXiv:2106.03338, 2021.

\bibitem{OSW22}
Tuomas Orponen, Pablo Shmerkin, and Hong Wang.
\newblock {K}aufman and {F}alconer estimates for radial projections and a
  continuum version of {B}eck's theorem.
\newblock Preprint, arXiv:2209.00348, 2022.

\bibitem{Shmerkin22b}
Pablo Shmerkin.
\newblock A nonlinear version of {B}ourgain's projection theorem.
\newblock {\em J. Eur. Math. Soc. (JEMS)}, online first, 2022.

\bibitem{Shmerkin22}
Pablo Shmerkin.
\newblock On the packing dimension of {F}urstenberg sets.
\newblock {\em J. Anal. Math.}, 146(1):351--364, 2022.

\bibitem{Wolff99}
Thomas Wolff.
\newblock Recent work connected with the {K}akeya problem.
\newblock In {\em Prospects in mathematics ({P}rinceton, {NJ}, 1996)}, pages
  129--162. Amer. Math. Soc., Providence, RI, 1999.

\end{thebibliography}

\end{document}